\documentclass[12pt,a4paper]{article}
\usepackage{microtype}
\usepackage[english]{babel}
\usepackage[dvips]{graphicx}
\usepackage{amsmath, amssymb,latexsym}
\usepackage{mathrsfs}
\usepackage{psfrag}
\usepackage[all]{xy}
\usepackage{color}
\usepackage{hyperref}
\usepackage{enumitem}
\usepackage{lineno}
\usepackage[T1]{fontenc}
\usepackage{mathtools,halloweenmath}

\newtheorem{prop}{Proposition}[section]
\newtheorem{theo}[prop]{Theorem}
\newtheorem{cor}[prop]{Corollary}

\newtheorem{defn}[prop]{Definition}
\newtheorem{rem}[prop]{Remark}
\newtheorem{rems}[prop]{Remarks}
\newtheorem{lem}[prop]{Lemma}

\newenvironment{proof}
{\begin{trivlist} \item[\hskip \labelsep {\bf Proof}\hspace*{3 mm}]}
	{\hfill$\Box$\end{trivlist}}

\begin{document}

%\linenumbers*
	
\title{On the geometry of holomorphic curves and complex surface}
\author{Amanda Dias Falqueto and Farid Tari}

\maketitle
\begin{abstract}
We investigate the geometry of holomorphic curves and complex surfaces from a singularity theory viewpoint.
We show that with the choice of the holomorphic metric, the families of functions and mappings that measure the contact between curves or surfaces with model objects are holomorphic. This allows the application of singularity theory techniques to pick up the geometry that is invariant under translations and complex rotations.
\end{abstract}
	
\noindent {\small {\bf Key--words}: holomorphic curves, complex surfaces, contact, curvature, evolute, singularities.}
	
\noindent {\small {\bf 2020 Mathematics Subject Classification:} 57R45, 53B99.}
	
\noindent \makebox [40mm]{\hrulefill}

%%%%%%%%%%%%%%%%%%%%%%%%%%%%%%%%%%%%%%%%%%%%%%%%%%%%%%%%%%%%%%%%
\section{Introduction}\label{sec:intro}

Singularity theory made extensive and deep contributions to the local differential geometry of submanifolds of the Euclidean and Minkowski spaces (see, for example, \cite{BG,Cipolla_Gilin,Damon_Giblin_Haslinger,IRRT,Izu,Porteous}). The geometry of the submanifold is captured by its contact with degenerate objects, such as those with zero or constant Gaussian curvature.
The contact is given by the singularity type of some mappings on the surface. These mappings come in natural families that generally are versal deformations, so they yield important properties of the submanifold. 

Our aim in this paper is to show how this singularity theory approach can be used to study the local differential geometry of holomorphic curves in $\mathbb C^2$ and $\mathbb C^3$, and of complex surfaces in $\mathbb C^3$. This, we believe, paves the way for the study of complex submanifolds in $\mathbb C^n$ from a singularity theory viewpoint.

Usually, the metric induced by the Hermitian inner product is used when studying geometric properties of submanifolds that are invariant under the action of the unitary group on $\mathbb C^n$ (see, for example, \cite{Nomizu&Smyth}). 
However, with this inner product, the families of functions and mappings of interest are not holomorphic. 
For instance, the Gauss map of a hypersurface is not holomorphic, and the family of height functions and orthogonal projections are also not holomorphic. Therefore, the results of singularity theory
cannot be used to capture the geometry of submanifolds of $\mathbb C^n$ derived from the Hermitian inner product.

Our key observation is to use the holomorphic inner product 
\begin{equation}\label{eq:holomorphicscprod}
\langle v,w\rangle =\sum_{i=1}^{n}v_iw_i,
\end{equation}
which turns $\mathbb C^n$ into a holomorphic Riemannian manifold with metric $ds^2=\sum_{i=1}^{n}dz_i^2$. Its isometry group is given by $E(n,\mathbb C)=
\mathbb C^n\rtimes O(n,\mathbb C)$, which is the semi-direct product of the group of translations with the complex orthogonal group.
More details on this holomorphic metric can be found in \cite{PessersVeken}. The authors in \cite{MetrixAlgGeom} also make use of the holomorphic metric to study metric problems in $\mathbb R^n$ using methods from complex algebraic geometry.

It is worth observing that there are non-zero vectors $v\in \mathbb C^n$ with zero length, that is $\langle v,v\rangle=0$. Such vectors are called isotropic vectors. Isotropic curves are curves with isotropic tangent vectors at all points. 
In the case $n=2$, isotropic curves are lines, and through each point pass two such lines. In the case $n=3$, there are isotropic curves that are not lines; see \cite[Section 1.12]{Struik}.  Struik also provided a classification of isotropic surfaces, that is, those whose tangent planes are isotropic at all points (equivalently, whose normal vector is isotropic). These surfaces are isotropic planes, isotropic cylinders, isotropic cones, or tangent surfaces to an isotropic curve \cite[Section 5.6]{Struik}.

Our work is set within the framework of generic geometry, which studies generic properties of submanifolds.
For generic regular holomorphic curves in $\mathbb C^n$, with $n=2,3$, which we deal with in Sections \ref{sec:planecurves} and \ref{sec:spacecurves}, we expect the points with isotropic tangent vectors to be isolated points along the curve. We call such points isotropic points. On the other hand, for generic surfaces in $\mathbb C^3$ (see \S\ref{sec:surfaces}), the points where the tangent planes are isotropic form a regular curve, which we call the isotropic locus of the surface.

The concepts of inflection points, vertices, and the evolute of a plane curve can be defined in terms of the contact of the curve with lines and circles. The holomorphic inner product \eqref{eq:holomorphicscprod} allows us to define the curvature of a curve and derive from it the concepts mentioned above (see Section~\ref{sec:spacecurves}). For space curves, we also introduce the notion of torsion, together with the Frenet–Serret frame and its corresponding differential equations (\S\ref{sec:spacecurves}).

We deal in \S \ref{sec:surfaces} with complex surfaces in $\mathbb C^3$. 
We define the Gauss map and show that its derivative is a symmetric operator. 
The concepts of the first and second fundamental forms can also be defined for such surfaces, and with that, all the concepts on the geometry of surfaces apply (see \cite{PessersVeken, Struik}).
For both the curve and surface cases, we define the families of height functions, distance squared  functions and orthogonal projections. All these families are holomorphic and singularity theory results can be used to derive geometric information about the curve and surface.

In \S\ref{sec:planecurves}, we show that the envelope of the normal lines to a plane curve, with respect to the Hermitian inner product, is empty, whereas the envelope of the normal lines, with respect to the holomorphic metric, is not (Theorem \ref{theo:EmptyHermeEnv}). This shows that the holomorphic metric is the natural choice for our study.
The envelope of the normal lines is precisely the evolute of the curve. We show that the evolute extends smoothly to a regular curve at isotropic points and is tangent to the original curve at those points.
On the other hand, the focal set of a surface in $\mathbb C^3$ extends smoothly to a regular surface along the isotropic locus and that it is tangent to the original surface along this curve.

It is worth emphasising that our study is not merely concerned with the complexification of real-analytic curves and surfaces. First, not all holomorphic curves and complex surfaces arise in this way. Second, real curves and surfaces do not possess isotropic points; consequently, phenomena such as the behaviour of the evolute and the focal set at isotropic points have no analogue in the real setting.

In \cite{JoitasSiersmaTibar}, the authors studied the local component of the bifurcation set of the distance-squared function on algebraic curves 
$X\subset \mathbb C^2$, including at points at infinity, called the 
\textit{ED discriminant} of 
$X$. They showed that this discriminant contains the evolute of 
$X$, as well as additional components that do not appear in the real case.

Regarding the generic singularities that arise in the families of functions and mappings considered here, these are determined by the dimensions of the corresponding parameter spaces. As a result, the types of singularities occurring in the complex setting are similar to those appearing in the real case.

Finally, when computing invariants of analytic curves and surfaces in Euclidean space, such as the number of inflection points and vertices accumulated at a curve singularity, or the number of umbilics accumulated at a degenerate umbilic point, one complexifies in order to obtain an upper bound for their real counterparts. The approach adopted in this paper assigns geometric meaning to the resulting complexified objects through the curvatures introduced here.

%%%%%%%%%%%%%%%%%%%%%%%%%%%%%%%%%%%%%%%%%%%%%%%%%%%%%%%%%%%%%%%%
\section{Preliminaries}\label{sec:prel}
We consider the holomorphic inner product in $\mathbb C^n$ defined in (1), which is complex bilinear, symmetric and non-degenerate form.

The holomorphic inner product (1) has a property similar to that of the Lorentzian inner product: there are non-zero vectors that satisfy  $\langle v,v\rangle=0$. These vectors are called {\bf isotropic} (\cite{Struik}).

Two non-zero vectors $v$ and $w$ are said to be {\bf orthogonal} if $\langle v,w\rangle=0$.

A hyperplane $\pi$ in $\mathbb C^n$ through some point $p_0$ has equation $\langle z-p_0,v\rangle=0$. The vector $v$ is called its normal vector. The hyperplane is {\bf isotropic} if its normal vector is isotropic. In that case, $v$ is parallel to $\pi$.

The {\bf length of a vector} $v\in \mathbb C^n$ is defined as $\langle v,v\rangle^\frac12$. This is a bi-valued set as it depends on the choice of the branch of the square root, that is,
$$
\langle v,v\rangle^\frac12=\left\{ 
| \langle v,v\rangle|^\frac12 e^{i \frac{\arg_{[0, 2\pi)} \langle v,v\rangle}2},\,\,
| \langle v,v\rangle|^\frac12 e^{i \frac{\arg_{(-\pi, \pi]} \langle v,v\rangle}2} 
\right\},
$$
where $\arg_I$ means the choice of the argument in the interval $I$. 

The case $I=(-\pi, \pi]$ 
is called the principal branch of the square root function, and the case $I=[0,2\pi)$ here is called the 
``other branch''. 

Let $C$ be a holomorphic curve in $\mathbb C^n$ parametrised locally by $\gamma: D\to \mathbb C^n$, where $\gamma$ is a holomorphic function and $D$ is a connected and simply connected open set.  Suppose that $\gamma'(t)$ is not an isotropic vector for all $t\in D$, and let 
\begin{align}\label{setB}
\mathcal{B}^- = \left\{ t \in D \; \middle| \; \operatorname{Re} (\langle \gamma'(t), {\gamma'(t)} \rangle ) < 0 \right\},\quad 
\mathcal{B}^+ = \left\{ t \in D \; \middle| \; \operatorname{Re} (\langle \gamma'(t), {\gamma'(t)} \rangle ) > 0 \right\}.
\end{align}

As $\gamma$ is holomorphic, $\mathcal{B}^{\pm}$ is a real semi-analytic subset of $D$, in particular, it has measure zero for generic curves.

To make the length of $\gamma'(t)$ in \S\ref{sec:planecurves} and \S\ref{sec:spacecurves} a single-valued function, 
we use the principal branch of the square root when $t\in D\setminus \mathcal{B}^-$ and the other branch when $t\in D\setminus \mathcal{B}^+$. We show that the geometric features derived from the functions used here are independent of the choice of branch of the square root.

A complex sphere of centre $p$ and radius $r\in \mathbb C$, $r\ne 0$, is defined as the set of points $z\in \mathbb C^n$ that satisfy
$$
\langle z-p, {z-p}\rangle=r^2.
$$  

The complex unit sphere $\mathbb CS^{n-1}$ is the set 
$$
\mathbb CS^{n-1}=\{z\in \mathbb C^n: \langle z, {z}\rangle=1 \},
$$
so every point in $\mathbb CS^{n-1}$ represents a unit vector.

When $n=2$, $\mathbb CS^{1}$ is the complex unit circle and has 
equation 
$
z_1^2+z_2^2=1,
$ 
 where $z=(z_1,z_2)\in \mathbb C^2$.
Complex circles are conics that passes through 
the circular points $(1:i:0)$  and $(1:-i:0)$ at infinity in $\mathbb{CP}^2$. These 
represent the two lines
$z_2=\pm iz_1$ in the affine plane $\mathbb C^2\subset \mathbb{CP}^2$.
The tangent vectors to these lines are along the isotropic vectors in $\mathbb C^2$.

We use the usual singularity theory concepts and notation about the actions of the Mather groups on the set of germs of holomorphic mappings (see, for example, \cite{wallsurvey}). 

Two germs of families of holomorphic functions $F,G: (\mathbb C^n\times \mathbb C^m,(0,0))\to (\mathbb C,0 )$
are said to be $\mathcal R^+$-equivalent if 
$$
G(z,v)=F(\phi(z,v),\psi(v))+c(v),
$$ 
where $(\phi,\psi):(\mathbb C^n\times \mathbb C^m,(0,0))\to (\mathbb C^n\times \mathbb C^m,(0,0))$ is a germ of a holomorphic diffeomorphism and $c:(\mathbb C^m,0)\to (\mathbb C,0)$ is a germ of a holomorphic function.

%%%%%%%%%%%%%%%%%%%%%%%%%%%%%%%%%%%%%%%%%%%%%%%%%%%%%%%%%%%%%%%%
\section{Plane curves}\label{sec:planecurves}

We define some basic concepts of regular holomorphic curves $C$, including curvature, vertices and evolute of such curves. We then consider the contact of  $C$ with lines and circles, as is done for real plane case in \cite{BG}.

All the concepts treated here are local in nature, so we take a local parametrisation 
 $\gamma: D\to \mathbb C^2$ of $C$ at a given point, where $\gamma$ is a holomorphic function and $D$ 
 is a connected and simply connected open set. We write $\gamma(t)=(z_1(t), z_2(t))$.

Suppose that $\gamma'(t)$ is not an isotropic vector. We choose a branch of the square root function 
and call the vector
$$
T(t)=\frac{\gamma'(t)}{\langle \gamma'(t), {\gamma'(t)}\rangle^{\frac12}}=\frac{(z_1'(t),z_2'(t))}{(z_1'(t)^2+z_2'(t)^2)^{\frac12}}
$$
the {\bf unit tangent vector} to $\gamma$ at $t$.

The unit vector 
$$
N(t)=\frac{(-z_2'(t),z_1'(t))}{(z_1'(t)^2+z_2'(t)^2)^{\frac12}} 
$$
is orthogonal to $T(t)$ and is called the {\bf unit normal vector} to $\gamma$ at $t$. 

The map $N:D\to \mathbb CS^1$ is called the {\bf Gauss map} of the curve $\gamma$. 

We say that $\gamma$ is a {\bf unit speed parametrisation} if $\langle \gamma'(t), {\gamma'(t)}\rangle=1$, that is, $\gamma'(t)\in \mathbb CS^1$, for all $t\in D$.

A point $\gamma(t)$ is an {\bf isotropic point} if $\gamma'(t)$ is an isotropic vector.

\begin{rem}
{\rm 
As the inner product (\ref{eq:holomorphicscprod}) is holomorphic, the isotropic points are isolated in $D$ or are the whole $D$. In the latter case, the curve is parallel to one of the isotropic lines $z_2=\pm iz_1$.
}
\end{rem}

\begin{theo}\label{theo:unitspeed}
Let $\gamma:D\to \mathbb C^n$ be a local parametrisation of a regular holomorphic curve. 
Suppose that $t_0\in D$ is not an isotropic point. Then there is a simply connected open subset $D'$ of $D$ containing $t_0$ such that $\gamma|_{D'}$ can 
be reparametrised by unit speed.
\end{theo}

\begin{proof}
The proof follows the same steps as that for curves in the real plane; however some necessary adjustment need to be made. Let 
$$
l(t)=\int_{t_0}^t \langle \gamma'(u), {\gamma'(u)}\rangle^\frac12 du,
$$
where the integral is taken along any path that joins $t_0$ and $t$ in $D\setminus \mathcal{B}^-$ or $D\setminus \mathcal{B}^+$ depending on the choice of the branch of the square root function. 
The function $l:D\to \mathbb C$ 
is holomorphic in some neighbourhood of $t_0$ where $\langle \gamma'(t), {\gamma'(t)}\rangle \ne 0$.

Since $l'(t_0)=\langle \gamma'(t_0), {\gamma'(t_0)}\rangle^{\frac 12} \ne 0$, it follows by the inverse function theorem that there exist a connected, simply connected open set $D'\subset D$ containing $t_0$ such that 
$l:D'\to l(D')$ is biholomorphic.

The local reparametrisation $\beta(s)=\gamma(l^{-1}(s))$ of the curve $\gamma$ from $l(D')\to \mathbb C^n$ is
unit speed since 
$$
\beta'(s)=\frac{\gamma'(l^{-1}(s))}{l'(l^{-1}(s))}=\frac{\gamma'(l^{-1}(s))}{\langle \gamma'(l^{-1}(s)), {\gamma'(l^{-1}(s))}\rangle^\frac12}.
$$
\end{proof}

\begin{rem}
{\rm In the real case, $l'(u)\ne 0$ for all $u$ in the interval of definition of $\gamma$, so $l$ is strictly monotonous. Therefore, it has an inverse $l^{-1}: l(I)\to I$. In the complex case, the biholomorphicity of $l$ is valid only locally at $t_0$.
}
\end{rem}

We now shrink $D$ if necessary and take $\gamma:D\to \mathbb C^2$ to be a unit speed local parametrisation. Differentiation $\langle T(s), {T(s)}\rangle=1$ gives $\langle T'(s), {T(s)}\rangle=0$. This means that $T'(s)$ is parallel to $N(s)$, so there exist a scalar function $\kappa(s)$ such that 
$$
T'(s)=\kappa(s)N(s).
$$

We call $\kappa(s)$ the {\bf curvature} of $\gamma$ at $s$.

\begin{prop}\label{prop:SeretFrenetPlabnecurves}
Let $\gamma:D\to \mathbb C^2$ be a unit speed local parametrisation of a holomorphic plane curve, with $\gamma(t)=(z_1(t),z_2(t))$. Then,
 
{\rm (1)} $\kappa(s)=\langle T'(s), {N(s)}\rangle=(z_1'z_2''-z_2'z_1'')(s)$.

{\rm (2)} $N'(s)=-\kappa(s)T(s).$
\end{prop}

\begin{proof}
Statement (1) is immediate.

For (2), differentiating $\langle T, {T} \rangle=z_1'^2+z_2'^2=1$ gives $z_1'z_1''+z_2'z_2''=0$.
That, with $z_1'z_2''-z_2'z_1''=\kappa$ yields $z_1''=-\kappa z_2'$ and $z_2''=\kappa z_1'$.

Differentiating $N=(-z_2',z_1')$, we get 
$N'=(-z_2'',z_1'')=-\kappa (z_1',z_2')=-\kappa T.$
\end{proof}

\begin{rems}\label{rem:kappaRegularC}
{\rm 
1. For a general parametrisation $\gamma$, one can reparametrise by unit speed, and show that the curvature is given by 
\begin{equation}\label{eq:kappaGeneral}
\kappa=\frac{z_1'z_2''-z_2'z_1''}{(z_1'^2+z_2'^2)^{\frac32}},  
\end{equation}
so,
\begin{equation}\label{eq:kappaDashGeneral}
\kappa'=\frac{(z_1'z_2'''-z_2'z_1''')(z_1'^2+z_2'^2)-3(z_1'z_1''+z_2'z_2'')(z_1'z_2''-z_2'z_1'')}{(z_1'^2+z_2'^2)^{\frac52}}.   
\end{equation}

2. The curve $\gamma$ is a part of a line if and only if $\kappa\equiv 0$.
It is a part of a complex circle if and only if $\kappa\equiv const\ne 0$.
}
\end{rems}

\begin{defn}\label{def:InfVert}
A point $t_0\in D$ on $\gamma$ is called an {\bf inflection of order $l$} if $\kappa(t_0)=\kappa'(t_0)=\ldots=\kappa^{(l-1)}(t_0)=0$ and $\kappa^{(l)}(t_0)\ne 0$. When $l=1$, it is called an 
ordinary inflection.

The point $t_0$ is called a {\bf vertex of order $l$} if $\kappa'(t_0)=\ldots=\kappa^{(l)}(t_0)=0$ and $\kappa^{(l+1)}(t_0)\ne 0$.  When $l=1$, it is called an 
ordinary vertex.
\end{defn}

\begin{rems}\label{rems:Propertieskappa}
{\rm 
1. It follows from the general expressions for the curvature in \eqref{eq:kappaGeneral} and its derivative in \eqref{eq:kappaDashGeneral} that the notions of ordinary inflection and vertex do not depend on the choice of the branch of the square root function. The same holds for inflections and vertices of any order, by induction on the expression of the derivative of the curvature function in \eqref{eq:kappaGeneral}.

1. Inflections of a curve are affine invariant and are defined as points where the curve has at least 3-point contact with its tangent line, equivalently, the intersection multiplicity of the curve with its tangent line is greater than 2. It is not difficult to show that this agrees with Definition \ref{def:InfVert} above (see \S \ref{ssec:contactlinesPlaneC}).

2. Vertices of plane curves can be defined in a similar way using the contact of the curve with complex circle, see \cite{BillMarcoF,Piene_etal,viro} and \S \ref{ssec:contactCircles}.
}
\end{rems}

We have the following obvious observation as a corollary of Proposition \ref{prop:SeretFrenetPlabnecurves}.

\begin{cor}
The Gauss map is a local bi-holomorphic diffeomorphism away from inflection points. It has an $A_1$-singularity at ordinary inflection points. 
\end{cor}

We next make an observation concerning the $H$-normal lines of a curve in $\mathbb C^2$ with respect to the Hermitian inner product 
$$
\langle v,w\rangle_H=v_1\overline{w}_1+v_2\overline{w}_2.
$$

The equation of the $H$-normal line at $t$ of a curve $\gamma$  is given by 
$$
G_t(p)=\langle \gamma(t)-p, \gamma'(t)\rangle_H=0, \,\, p\in \mathbb C^2.
$$

The family of functions $G(t,p)=G_t(p)$ is not holomorphic, so we cannot define the envelope as a discriminant (\cite{BillEnvelopes}). 
We identify $\mathbb C^2$ with $\mathbb R^4$ and use the following definition of an envelope from \cite{ThomEnv}; see also \cite{MarioJorge}. 
A  germ of a $q$-parameter family of submanifolds of codimension $m-n+q$ in $\mathbb R^m$ is a 
diagram of smooth map-germs of the form 
$$
\mathbb R^q,0 \xleftarrow{\pi} \mathbb R^n,0 \xrightarrow{f}\mathbb R^m,0
$$
that satisfies the following conditions:

(a) $\pi$ is a germ of a fibration;

(b) $f$ restricted to each fibre $\pi^{-1}(q)$ is a germ of a one to one immersion.

If $S(f)=\{p\in \mathbb R^n: df_p \, \mbox{ is not surjective}\}$, then $E=f(S(f))$ is the envelope of the family.

\begin{theo} \label{theo:EmptyHermeEnv}
Let $\gamma:D\to \mathbb C^2$  be a parametrisation of a regular holomorphic curve. Then,
the envelope of the family of the normal lines to $\gamma$, with respect to the Hermitian inner product,  is empty.
\end{theo}

\begin{proof} We take $\gamma$ a unit speed local parametrisation (away from isotropic points) and consider the following diagram
$$
\mathbb C \xleftarrow{\pi} \mathbb C\times \mathbb C \xrightarrow{f}\mathbb C^2,
$$
where $\pi(s,v)=v$ and $f(s,v)=\gamma(s)+v \gamma'(s)^{\perp}$, where $\gamma'(s)^{\perp}=(-\overline{z'_2(s)},\overline{z'_1(s)})$ is 
orthogonal to $\gamma'(s)$ with respect to the Hermitian inner product. Then, the image of the fibre $\pi^{-1}(s)$ is  the $H$-normal line to $\gamma$ at $s$. 

We consider the family of the $H$-normal lines to $\gamma$ as a family of planes in $\mathbb R^4$ and show that their envelope is empty. 

We write $s=s_1+is_2$, $v=v_1+iv_2$, and  $\gamma(s)=(z_1(s),z_2(s))$.  Using the 
Wirtinger derivatives, we get
$$
\begin{array}{lcl}
\frac{\partial f}{\partial s_1}=(z_1', z_2')+v(-\overline{z_2''},\overline{z_1''}),&\quad&
\frac{\partial f}{\partial v_1}=(-\overline{z_2'},\overline{z_1'}),\\[0.2cm]
\frac{\partial f}{\partial s_2}=(z_1', z_2')+v(i\overline{z_2''},-i\overline{z_1''}),&\quad&
\frac{\partial f}{\partial v_2}=i(-\overline{z_2'},\overline{z_1'}).
\end{array}
$$

Identifying $\mathbb C^k$ with $\mathbb R^{2k}$, the determinant of the Jacobian matrix of $df_{(s,v)}$ is $1+||v\kappa(s)||^2>0$. Therefore, 
$S(f)$ is empty. Consequently, the $H$-normal lines to $\gamma$ do not have an envelope.
\end{proof}
%%%%%%%%%%%%%%%%%%%%%%%%%%%%%%%%%%%%%%%%%%%%%%%%%%

\subsection{Contact with lines}\label{ssec:contactlinesPlaneC}

A line in $\mathbb C^2$, with an orthogonal vector $v$, has equation 
$$
\langle z,{v}\rangle =c,
$$
for some constant $c\in \mathbb C$.

In view of Remarks \ref{rems:Propertieskappa} (1), we 
consider $t$ in  the entire domaine of the parametrisation of the curve and define the {\bf family of height functions} $H:D\times \mathbb CS^1\to \mathbb C$
on $\gamma$ by
$$
H(t,v)=\langle \gamma(t),{v}\rangle.
$$ 

The function $H$ is holomorphic. For $v$ fixed, the function  $H_v(t)=H(t,v)$ is the height function along $v$.
Its singularities measure the contact of the curve $\gamma$ with the lines orthogonal to $v$.

We have the following result; see \cite{BG} for the real case analogue. 

\begin{theo}
	Let $\gamma:D\to \mathbb C^2$ be a local  parametrisation of a regular holomorphic curve without isotropic points.
	Then,
	
	{\rm (1)} The function $H_v$ is singular at $t_0$ if and only if $v\parallel N(t_0)$.
	
	{\rm (2)} For $v\parallel  N(t_0)$, the singularity of $H_v$ at $t_0$ is of type $A_1$ if and only if $\kappa(t_0)\ne 0$.
	It is of type  $A_2$ if and only if $\kappa(t_0)=0$ and $\kappa'(t_0)\ne 0$, that is, $t_0$ is an ordinary inflection point of the curve.
	
	{\rm (3)} For generic curves, $H_v$ can only have an $A_1$ or $A_2$ singularity and these are $\mathcal R^+$-versally unfolded by the family $H$.
\end{theo}

\begin{rem}
{\rm 
The family of height functions can be defined at an isotropic point $t_0$ 
as $H:D\times \mathbb C P^1\to \mathbb C$, with $H(t,\bar{v})=\langle \gamma(t),{v}\rangle$ and $v\in \mathbb C^2$ any representative of $\bar{v}$. Then $H_{\bar{v}}$ is singular at $t_0$ if and only if $\bar{v}\parallel \gamma'(t_0)$. The singularity is generically an $A_1$ as generically $\langle \gamma''(t_0),\gamma'(t_0)\rangle\ne 0.$
}
\end{rem}

%%%%%%%%%%%%%%%%%%%%%%%%%%%%%%%%%%%%%%%%%%%%%%%%
\subsection{Contact with circles}\label{ssec:contactCircles}

The contact of the curve $\gamma$ at $t_0$ with the circle of centre ${c}$ passing through $\gamma(t_0)$ is measured by the singularities of the (contact) function 
\begin{equation}
d_{c}(t)=\langle \gamma(t)-c, {\gamma(t)-c}\rangle. 
\end{equation}

The function $d_c$ is holomorphic, and we call it the {\bf distance squared function}. 

Varying $c\in \mathbb C^2$, yields the {\bf family distance squared functions}
$d:D\times  \mathbb C^2\to \mathbb C$, given by 
\begin{equation}
d(t,c)=d_c(t),
\end{equation}
which is also a holomorphic function.

Away from isotropic points, we parametrise 
$\gamma$ by unit speed, using one of the branches of the square root function. 
Then, the function $d_c$ is singular at $s_0$ (i.e., has an $A_{\ge 1}$-singularity) if and only if
$$
\frac12 d'_{c}(s_0)=\langle T'(s_0), {\gamma(t_0)-c}\rangle =\langle \gamma(t_0)-c, {T'(s_0)}\rangle=0,
$$
that is, 
 $c=\gamma(s_0)+\lambda N(s)$, 
for some $\lambda \in \mathbb C$. 
This means that the circle is tangent to the curve at $s_0$ if and only if its centre ${c}$ 
belongs to the line through $\gamma(s_0)$ and parallel to $N(s_0)$.
We call this line the {\bf normal line} to $\gamma$.

Differentiating twice, gives
$$
\frac12 d''_{c}(s_0)=\kappa(s_0)\langle N(s_0), {\gamma(s_0)-c}\rangle +1.
$$

Suppose that $s_0$ is not an inflection point. 
Then, the singularity of $d_c$ at $s_0$ is of type $A_{\ge 2}$ if and only if $d'_c(s_0)=d''_c(s_0)=0$, equivalently, 

\begin{align} \label{eq:evolute}
c=e(s_0)=&\,  \gamma(s_0)+\frac1{\kappa(s_0)}N(s_0) \nonumber \\
=&\,  (z_1(s_0),z_2(s_0))+\frac{(z_1'^2+z_2'^2)(s_0)}{(z_1'z_2''-z_2'z_1'')(s_0)}(-z'_2(s_0),z'_1(s_0)).
\end{align}

Varying $s_0$ in $D$, we get a parametrised plane curve, called the {\bf evolute} of $\gamma$. 
Clearly, the evolute is independent of the choice of branch of the square root function.

The circle of centre $e(s_0)$ tangent to the curve at $\gamma(t_0)$ is called the {\bf osculating circle} of $\gamma$ at $s_0$.

With $c=e(s_0)$, we have $\frac12 d'''_c(s_0)=0$ if and only if $\kappa'(s_0)=0$. Thus, the singularity of 
$d_c$ at $t_0$ is of type $A_{\ge 3}$ if and only if $c=e(s_0)$ and $t_0$ is a vertex of $\gamma$.

We have the following result;  see \cite{BG} for the real case analogue. 

\begin{theo}
Let $\gamma:D\to \mathbb C^2$ be a local parametrisation of a regular holomorphic curve without inflections and isotropic points. Then,
the distance squared functions $d_c$ has a singularity of type

$A_1$ $\iff$ $c\ne e(t_0)$ is on the normal line of $\gamma$ at $t_0$.

$A_2$ $\iff$ $c= e(t_0)$ and $t_0$ is not a vertex of $\gamma$.

$A_3$ $\iff$ $c= e(t_0)$ and $t_0$ is an ordinary vertex of $\gamma$.

For a generic $\gamma$, the function $d_c$ has only the above singularities and these are $\mathcal R^+$-versally unfolded by the family $d$. 

The local component of the bifurcation set of the family $d$ is precisely the evolute of $\gamma$. 
Consequently, the evolute of a generic curve can only have cusp singularities.
\end{theo}

As $e'(t)=-\frac{\kappa'(t)}{\kappa(t)^2}N(t)$, we have the following consequence.

\begin{cor}\label{cor:envelopNormals}
Away from vertices, the tangent line to the evolute is the normal line to the curve at the associated point.
Consequently, the evolute is the envelope of the normal lines of $\gamma$.
\end{cor}

The following result describes the behavior of the evolute at isotropic points.

\begin{prop}
Let $\gamma: D \rightarrow \mathbb{C}^{2}$ be a regular holomorphic curve. Suppose that $t_{0}\in D$ is an isotropic point and that the curve has an ordinary contact with its tangent line at $t_0$. Then, the evolute extends to a holomorphic curve at $t_0$ that is tangent to $\gamma$ at this point.
\end{prop}

\begin{proof}
The proof follows from the parametrisation of the evolute in \eqref{eq:evolute}. The assumption on the curve $\gamma$ having an ordinary contact with its tangent line at $t_0$ is equivalent to $(z_1'z_2''-z_2'z_1'')(t_0)\ne 0.$
\end{proof}

%%%%%%%%%%%%%%%%%%%%%%%%%%%%%%%%%%%%%%%%%%
\subsection{Algebraic curves}
Let $f(z_1,z_2)=0$ 
be a regular algebraic curve of degree $d$  in $\mathbb{C}^2$. 
A general $f=0$ has $3d(d-2)$ inflection points and  $2d(3d-5)$ vertices; 
see \cite[Chapter 6]{MetrixAlgGeom} for proofs and references. 

Isotropic points of the curve $f=0$ are solutions of the system
\begin{equation}\label{sys:NrIP}
 \left\{
\begin{aligned}
f(z_1,z_2)&=0,\\
(f_{z_{1}}^{2} + f_{z_{2}}^{2})(z_{1},z_{2})&=0.
\end{aligned}
\right.
\end{equation}

We write 
$$
f(z_1,z_2)=f_d(z_1,z_2) + f_{d-1}(z_1,z_2) + \cdots + f_{0}(x,y),
$$
where each $f_j(z_1,z_2)$ is homogeneous of degree $j$.

\begin{theo} Let $f=0$ be a regular algebraic curve of degree $d$.
Suppose the circular points at infinity $(1: \pm i:0)$ are not roots of $f_d$, and that $f_d$ is reduced. Then the curve $f=0$ has $2d(d-1)$ isotropic points. 
\end{theo}

\begin{proof}
To count the number of solutions of the systems \eqref{sys:NrIP}, we projectivise and apply Bézout’s theorem to obtain the number of intersections of the corresponding curves in 
$\mathbb CP^2$. We then consider possible solutions that lie on the line at infinity.

Denote by 
$F(z_1:z_2:z_3)=0$ the homogenization of $f(z_1,z_2)=0$ in $\mathbb{C}P^2$, given by
\begin{equation}\label{eq:HomogF}
F(z_1:z_2:z_3) = f_d(z_1,z_2) + z_3 f_{d-1}(z_1,z_2) + \cdots + z_3^df_{0}(x,y).    
\end{equation}

Projectivising the second equation in \eqref{sys:NrIP}, we count the intersections of the curves  
$$
F=0 \quad \mbox{and}\quad (F_{z_{1}} + iF_{ z_{2}})(F_{z_{1}} -iF_{ z_{2}})=0
$$ 
in $\mathbb{PC}^{2}$. 

The isotropic points on the line at infinity are given by
$$\left\{
\begin{array}{rcl}
    f_{d} &=& 0,\\
    (f_{{d}_{z_{1}}} +if_{{d}_{z_{2}}})(f_{{d}_{z_{1}}} - if_{{d}_{z_{2}}})&=& 0.
\end{array}
\right.
$$

Let $\alpha = (\alpha_1:\alpha_{2}:0) \in \mathbb{PC}^{2}$ be a solution of the system of equations above. It follows from Euler's identity that 
$\alpha_1f_{{d}_{z_{1}}}(\alpha) + \alpha_{2}f_{{d}_{z_{2}}}(\alpha)=df_{d}(\alpha) = 0$. Substituting in the second equation of the system above yields $(\pm i\alpha_1 -\alpha_2)f_{{d}_{z_{2}}}(\alpha)=0$. Therefore, either $\alpha$ is a circular point at infinity or $f_d$ has a repeated factor. As both are excluded by hypothesis, it follows that none of the intersections of the curves lies on the line at infinity.
 
Now, $F_{z_{1}} \pm  iF_{ z_{2}}=0$ are of degree $d-1$ and $F=0$ is regular, so the curves $F=0$ and  $F_{z_{1}}^2 + F_{ z_{2}}^2=0$ have no common component and the result follows by Bézout's theorem.
\end{proof}
%%%%%%%%%%%%%%%%%%%%%%%%%%%%%%%%%%%%%%%%%%%%%%%%%%%%%%%%%%%%%%%%

\section{Space curves}\label{sec:spacecurves}

Let $\{e_1,e_2,e_3\}$ be the standard basis in $\mathbb C^3$. The cross product of two vectors 
$z=(z_1,z_2,z_3)$ and $w=(w_1,w_2,w_3)$ is the vector
$$
z\times w=\left|
\begin{array}{ccc}
e_1&e_2&e_3\\
z_1&z_2&z_3\\
w_1&w_2&w_3
\end{array}\right|=(z_2w_3-z_3w_2,z_3w_1-z_1w_3,z_1w_2-z_2w_1).
$$

The \textbf{osculating plane} of a regular holomorphic space curve $\gamma:D\to\mathbb C^3$ at a point $t_0\in D$ is the plane through $\gamma(t_0)$ parallel to $\gamma'(t_0)$ and $\gamma''(t_0)$.  
This osculating plane is isotropic if and only if $\gamma'(t_0)\times \gamma''(t_0)$ is isotropic, equivalently,
$$
\langle \gamma'(t_0),\gamma'(t_0)\rangle \langle \gamma''(t_0),\gamma''(t_0)\rangle-
\langle \gamma'(t_0),\gamma''(t_0)\rangle^2=0.
$$

We have the following observation.
\begin{lem}\label{lem:isotPlane}
Suppose that $\gamma$ is parametrised by unit speed at a non-isotropic point $s_0$, and let $T$ be the unit vector of $\gamma$. 
Then the osculating plane $\pi$ at $s_0$ is isotropic 
if and only if $T'(s_0)$ is isotropic.
\end{lem}

\begin{proof}
Assume first that $\pi$ is isotropic. Then $\pi$ has an isotropic normal vector 
$v\nparallel T(s_0)$, so it is generated by $T(s_0)$ and
$v$, with $\langle T(s_0),v\rangle=0$. 

As $T'(s_0)$ is parallel to $\pi$, we can write $T'(s_0)=\alpha T(s_0)+\beta v$ for some $\alpha,\beta\in \mathbb C$.
Differentiating the identity $\langle T,{T}\rangle=1$ gives 
 $\langle T',T\rangle=0$, hence $\alpha=0$. Thus, $T'(s_0)=\beta v$ 
 which is isotropic. 
(We are assuming $\pi$ to be an isotropic plane, so $\beta\ne 0$.)
 
Conversely, suppose that $T'(s_0)$ is an isotropic vector. For any 
$\alpha,\beta\in \mathbb C$, we have $\langle \alpha T(s_0)+\beta T'(s_0),T'(s_0)\rangle=0$.
Thus, $\pi$ is an isotropic plane with normal $T'(s_0)$.
\end{proof}

Struik (\cite[Section 1–12]{Struik}) describes space curves that are isotropic at every point.
He also presents a proof of a theorem, attributed to E. Study, stating that if the osculating plane of a space curve is isotropic at every point, then the curve is either an isotropic curve or lies entirely in an isotropic plane.

In this section, we consider holomorphic space curves without isotropic points and whose osculating plane is nowhere isotropic. We parametrise the curve by unit speed (Theorem \ref{theo:unitspeed}).

Under these assumptions, and by Lamma \ref{lem:isotPlane}, there exist a scalar function $\kappa$ and a unit vector $N$ such that
$$
T'(s)=\kappa(s) N(s).
$$

Observe that  $\kappa=\langle T',{N}\rangle$ is a holomorphic function, and $\kappa(s)\ne 0$ for all $s\in D$.

From $\langle T'(s),{T(s)}\rangle=0$, we get
$\kappa(s)\langle N(s),{T(s)}\rangle=0$. Therefore, 
$$
\langle N(s),{T(s)}\rangle=0.
$$

It follows that $T$ and $N$ are orthogonal. We call the vector $N(s)$ the {\bf unit normal vector} to $\gamma$ at $s$ and call $\kappa(s)$ the {\bf curvature} of $\gamma$ at $s$.

We define the {\bf binormal vector} $B(s)$ of $\gamma$ at $s$ as the vector 
$
B(s)=T(s)\times N(s)$. 

It follows from the properties of determinants that the vectors $T(s),N(s),B(s)$ form an orthonormal basis of $\mathbb C^3$, for all $s$. 
We call the frame $\{T,N,B\}$ the \textbf{Frenet-Serret frame} of $\gamma$.

We have $B'=T'\times N+T\times N'= T\times N'$, so 
$\langle B',{T}\rangle=\langle T\times N',{T}\rangle=0.$

Writing $B'=\lambda_1 T+\lambda_2 N+\lambda_3 B$, it follows 
from $\langle B',{T}\rangle=0$, and from the orthogonality of the vectors of the frame, that $\lambda_1=0$.

From $\langle B,B\rangle=1$, we get $\langle B',B\rangle=0$. 
It follows that
$\langle \lambda_2 N+\lambda_3 B,B\rangle=\lambda_3=0$. Therefore, $B'$ is parallel to $N$. We write
$$
B'(s)=\tau(s) N(s),
$$
and call $\tau(s)$ the {\bf torsion} of $\gamma$ at $s$. Clearly, $\tau=\langle B',{N}\rangle$ is a holomorphic function.

One can show that $T=N\times B$ and $N=B\times T$.
Then, 
$$
N'=B'\times T+B\times T'=-\tau B-\kappa T.
$$ 

\begin{prop} The Frenet-Serret formulae for the moving frame $T,N,B$ are:
$$
\left\{ \begin{array}{rcl}
T'&=&\kappa N,\\
N'&=&-\kappa T-\tau B,\\
B'&=&\tau N.\\
\end{array}
\right.
$$	
\end{prop}

Of course one can define the Frenet-Serret frame, the curvature and torsion for regular holomorphic curves not parametrised by unit speed, assuming that they are without isotropic points and that the osculating plane is not isotropic at all points.
Then,  
$$
\kappa=\dfrac{\langle \gamma' \times \gamma'', {\gamma' \times \gamma''}\rangle^{\frac{1}{2}}}{\langle\gamma', {\gamma'}\rangle^{\frac{3}{2}}}   
\quad \mbox{and}  \quad 
\tau=-\frac{\langle \gamma' \times \gamma'', {\gamma'''}\rangle}{\langle \gamma' \times \gamma'', {\gamma' \times \gamma''}\rangle}.
$$

\begin{rem}\label{rem:tors}
{\rm 
It is worth observing that the torsion is a holomorphic function and does not dependent on the branch of the square root.
}
\end{rem}
%%%%%%%%%%%%%%%%%%%%%%%%%%%%
\subsection{Contact with planes}\label{ss:contactPlanesCurves}

For a space curve $\gamma:D\to \mathbb C^3$, without isotropic points and with the assumption that the osculating plane is not isotropic at all points, 
we define the family of height functions on $\gamma$ as the holomorphic function
$H:D\times \mathbb CS^2\to \mathbb C$, where
$$
H(t,v)=\langle \gamma(t),{v}\rangle.
$$

We take, for simplicity, $\gamma$ to be a unit speed local parametrisation. 
Then, as shown above, $\kappa(s)\ne 0$ and  $T,N,B$ form an orthonormal frame.

The  height function along $v$ is defined as $H_v(s)=H(s,v)$ and measures the contact of the curve $\gamma$ with planes  orthogonal to $v$.

We have $H_v'=\langle T,{v}\rangle$, and it vanishes at $s_0$ if and only if $\langle v,{T(s_0)}\rangle=0$, that is, $v=\lambda N(s_0)+\mu B(s_0)$, for some $\lambda, \mu\in \mathbb C$, with $\lambda^2+\mu^2=1$.

Now $H_v''=\kappa \langle N,{v}\rangle$, so $H_v'(s_0)=H_v''(s_0)=0$ if and only if $v=\pm B(s_0)$. 

Differentiating again, we get $H_v'''=\kappa' \langle N,{v}\rangle-\kappa 
\langle \kappa T+\tau B,{v}\rangle$. Thus, $H_v'(s_0)=H_v''(s_0)=H_v'''(s_0)=0$ if and only if 
 $v=\pm B(s_0)$ and $\tau(s_0)=0$. With these conditions, one can show that  $H_v^{(4)}(s_0)\ne 0$ if and only if  $\tau'(s_0)\ne 0$

We call the plane orthogonal to $T(s)$ the {\bf normal plane} of $\gamma$ at $\gamma(s)$. We have the  following result (see \cite{BG} for real curves analogue).

\begin{theo}
Let $\gamma:D\to \mathbb C^3$ be a local parametrisation of a generic regular holomorphic curve without isotropic points and with the 
osculating plane not isotropic at all points.
Then, the height function $H_v$ can only have singularities of type $A_1,A_2,A_3$ and these occur at a given point $t_0$ when:

$A_1$ $\iff$ $v$ belongs to the normal plane of $\gamma$ at $t_0$ and $v \nparallel (\gamma'\times \gamma'')(t_0)$.

$A_2$ $\iff$ $v\parallel (\gamma'\times \gamma'')(t_0)$ and $\tau(t_0)\ne 0.$

$A_3$ $\iff$ $v\parallel (\gamma'\times \gamma'')(t_0)$, $\tau(t_0)=0$ and $\tau'(t_0)\ne 0$.

The above singularities of $H_v$ are $\mathcal R^+$-versally unfolded by the family $H$.
\end{theo}

\begin{rem}
{\rm 
Here too one can consider the contact of a space curve with isotropic planes 
by defining locally the family of height functions 
as $H:D\times \mathbb C P^2\to \mathbb C$, with $H(t,\bar{v})=\langle \gamma(t),{v}\rangle$ and $v\in \mathbb C^3$ any representative of $\bar{v}$.
}
\end{rem}

%%%%%%%%%%%%%%%%%%%%%%%%%%%%%%%%
 
\subsection{Contact with lines}\label{ss:contactLinesCurves}

We defined the {\bf orthogonal projection} along $v\in \mathbb CS^2$ as the parallel projection along $v$ to the plane through the origin and orthogonal to $v$.

The orthogonal projection $P_v(p)=p+\lambda v$ of a point $p\in \mathbb C^3$ along $v$ satisfies 
$\langle p+\lambda v,{v}\rangle=0$, so $\lambda =-\langle p,{v}\rangle$.

Observe that the  orthogonal plane to  $v$ can be identified with the tangent plane to the unit circle $\mathbb CS^2$ at $v$. 
By varying 
$v$ in $\mathbb CS^2$, we obtain the {\bf family of orthogonal projections} $P:\mathbb C^3\times \mathbb CS^2\to T\mathbb CS^2$, given by
$$
P(p,v)=(v,P_v(p))=(v, p-\langle p,{v}\rangle v).
$$

Clearly, $P$ is a holomorphic map. The family of orthogonal projections of a space curve $\gamma$ is the restriction of $P$ to $\gamma$. 

We take $\gamma$ without isotropic points and with the 
osculating plane not isotropic at all points. We also assume for simplicity that it is a unit speed parametrisation. 

The orthogonal projection $P_v(s)=\gamma(s)-\langle \gamma(s),{v}\rangle v$ 
can be considered locally at $s_0\in D$ as a map-germ $(\mathbb C,s_0)\to (\mathbb C^2, P_v(\gamma(s_0)))$. 

We have
$$
P_v'(s)=T(s)-\langle T(s),{v}\rangle v.
$$

This means that $P_v$ is singular at $s_0$ if and only if $v=\pm T(s_0)$. Differentiating again gives
$$
P_v''(s)=\kappa(s)(N(s)-\langle N(s),{v}\rangle v).
$$

Therefore, at a singularity $s_0$ of $P_v$, we have $P_v''(s_0)=\kappa(s_0)N(s_0)$. It follows that the singularity of the defining equation of $P_v(\gamma)$ is of type $A_k$ (since $\kappa(s_0)\ne 0$). 

We have 
$$
P_v'''(s)=\kappa'(s)[N(s)-\langle N(s),{v}\rangle v] -\kappa(s)[\kappa(s)T(s)+\tau(s)B(s)-\langle \kappa(s)T(s)+\tau(s)B(s),{v}\rangle v],
$$
so at a singularity $s_0$ of $P_v$ we have
$$
P_v'''(s_0)=\kappa'(s_0)N(s_0) -\kappa(s_0)\tau(s_0)B(s_0).
$$

The vectors $P_v''(s_0)$ and $P_v'''(s_0)$ are linearly independent if and only if $\tau(s_0)\ne 0$.
In that case, the singularity of $P_v$ at $s_0$ is $\mathcal A$-equivalent to the cusp $s\mapsto (s^2,s^3)$ (i.e., it is of type $A_2$).

With  $v=\pm T(s_0)$ and $\tau(s_0)=0$,  we have 
$$
\begin{array}{rcl}
P_v^{(4)}(s_0)&=&(\kappa''(s_0)-\kappa^3(s_0))N(s_0) -\kappa(s_0)\tau'(s_0)B(s_0),\\
P_v^{(5)}(s_0)&=&(\kappa'''(s_0)-6\kappa'(s_0)\kappa^2(s_0))N(s_0) -(3\kappa'(s_0)\tau'(s_0)+\kappa(s_0)\tau''(s_0))B(s_0).
\end{array}
$$

The singularity of $P_v$ is ${\mathcal A}$-equivalent to $s\mapsto  (s^2,s^5)$ (i.e., it is of type $A_4$) if and only if $(\kappa'\tau'-3\kappa\tau'')(s_0)\ne 0$.

We have the following results (see \cite{David} for the real case analogue).

\begin{theo}\label{theo:ProjSpCurves}
	Let $\gamma:D\to \mathbb C^3$ be a local parametrisation of a generic regular holomorphic curve without isotropic points and with the 
osculating plane not isotropic at all points.
	Then, the orthogonal projection $P_v$ along $v$ can only have local singularities of type $A_2$ or 
	$A_4$. We get, at $t_0$, a singularity of type 
	
	$A_2$ $\iff$ $v$ is tangent to $\gamma$ at $t_0$ and $\tau(t_0)\ne 0.$
	
	$A_4$ $\iff$ $v$ is tangent to $\gamma$ at $t_0$, $\tau(t_0)=0$ and $(\kappa'\tau'-3\kappa\tau'')(t_0)\ne 0$.
	
	The above singularities are $\mathcal A_e$-versally unfolded by the family $P$.
\end{theo}

\begin{rem}
{\rm 
The torsion is a holomorphic function and does not depend on the choice of the branch of the square root (Remark \ref{rem:tors}). 
The numerator of the function $\kappa'\tau'-3\kappa\tau''$ is holomorphic, so the results in Theorem \ref{theo:ProjSpCurves} do not depend on the choice of the branch of the square root and do not require the curve to be parametrized by unit speed.
}
\end{rem}

%%%%%%%%%%%%%%%%%%%%%%%%%%%%%%%%%%%%%%%%%%%%%%%%%%%%%%%%%%%%%%%%

\section{Surfaces in $\mathbb{C}^3$} \label{sec:surfaces}

Let $M$ be a surface in $\mathbb C^3$, that is, a complex two-dimensional submanifold of $\mathbb C^3$. We study the extrinsic geometry of $M$ and work locally at a given point on $M$. We take a  
local parametrisation $\phi: U\to \mathbb C^3$ of $M$ at that point. 
The map $\phi$ is a regular holomorphic map, and we write $M=\phi(U)$. 

The {\bf first fundamental form} of $M$ at $p=\phi(q)\in M$ is the quadratic form 
$I_p(v)=\langle v,{v}  \rangle$, for all $v\in T_pM$.
We denote 
$$
E=\langle \phi_{z_1},{\phi_{z_1}} \rangle,\quad
F=\langle\phi_{z_1},{\phi_{z_2}}\rangle,\quad
G=\langle \phi_{z_2},{\phi_{z_2}}\rangle,
$$
the coefficients of the first fundamental form $I$. Then, for any $v=a\phi_{z_1}(q)+b\phi_{z_2}(q)\in T_pM$, we have 
$$
I_p(v)=a^2E+2ab F+b^2G.
$$

There may exist points on the surface where the quadratic form $I_p$ is degenerate, that is, points $p=\phi(q)$ such that
$$
\delta(q)=(EG-F^2)(q)=0.
$$ 
We call the set of such points the \textbf{isotropic locus ($IL$)} of 
$M$, and the points on this set are called \textbf{isotropic points}.
We identify the $IL$ with its preimage in $U$, so that
$$
IL=\{(z_1,z_2)\in U: (EG-F^2)(z_1,z_2)=0\}.
$$

In fact, a point $p$ is isotropic if and only if the tangent plane $T_pM$ is isotropic. Surfaces whose points are all isotropic are the isotropic planes, isotropic cylinders, isotropic cones, or tangent surfaces to isotropic curves (see \cite[Sections 5–6]{Struik}).

As pointed out previously, our aim is to study the geometry of generic submanifolds, so we make the following assumption.

\medskip
{\bf Assumption.} We assume that the isotropic locus $IL$ of $M$ is either empty or consists of a regular curve on $M$.

\medskip

At points $p=\phi(z_1,z_2)$ on $M\setminus IL$, the vector $\phi_{z_1}\times \phi_{z_2}$ is non-isotropic and  orthogonal to $T_pM$. 
We shrink $U$ if necessary, choose a branch of the square root function and define the
{\bf Gauss map} $N:M\setminus IL\to \mathbb CS^2$,
by 
$$
N(z_1,z_2)=\frac{\phi_{z_1}\times \phi_{z_2}}{\langle \phi_{z_1}\times\phi_{z_2}, {\phi_{z_1}\times\phi_{z_2}}\rangle ^{\frac12}}(z_1,z_2).
$$ 

As in the case of curves, all the geometric features of the surface derived from the Gauss map $N$ are independent of the choice of the branch of the square root function used in the definition of $N$.

Differentiating the identity $\langle N, {N}\rangle =1$, gives 
$\langle N_{z_1}, {N}\rangle =\langle N_{z_2}, {N}\rangle =0$. It follows that 
$-dN_p:T_pM\to T_{N(p)}\mathbb CS^2\simeq T_pM$ is a linear operator on $T_pM$, called the {\bf shape operator} of the surface.

From $\langle \phi_{z_1}, {N}\rangle = \langle \phi_{z_2}, {N}\rangle=0$, we get 
$$
\begin{array}{rcl}
\langle \phi_{z_1z_2}, {N}\rangle+ \langle \phi_{z_1}, {N_{z_2}}\rangle&=&0,\\
\langle \phi_{z_2z_1}, {N}\rangle+ \langle \phi_{z_2}, {N_{z_1}}\rangle&=&0.
\end{array}
$$

It follows that $\langle \phi_{z_1}, {N_{z_2}}\rangle=\langle \phi_{z_2}, {N_{z_1}}\rangle$, so  
$\langle dN_p(\phi_{z_1}),\phi_{z_2}\rangle =\langle dN_p(\phi_{z_2}),\phi_{z_1}\rangle$. It follows that 
 the bilinear form on $T_pM$, given by
$$
II_p(u,v)=\langle -dN_p(u), {v}\rangle,
$$
is symmetric. The quadratic form 
$II_p(v,v)$ is called the {\bf second fundamental form} of $M$, and is denoted $II_p(v)$. Its coefficients are given by
$$
\begin{array}{ccccccc}
	l&=&-\langle N_{z_1},\phi_{z_1}\rangle&=& \langle N,\phi_{z_1z_1}\rangle&=&\frac{\overline{l}}{\langle {\phi_{z_1}\times\phi_{z_2}}\rangle ^{\frac12}},\\
	m&=&-\langle N_{z_1},\phi_{z_2}\rangle&=& \langle N,\phi_{z_1z_2}\rangle&=&\frac{\overline{m}}{\langle{\phi_{z_1}\times\phi_{z_2}}\rangle ^{\frac12}},\\
	n&=&-\langle N_{z_2},\phi_{z_2}\rangle&=& \langle N,\phi_{z_2z_2}\rangle&=&\frac{\overline{n}}{\langle{\phi_{z_1}\times\phi_{z_2}}\rangle ^{\frac12}},
\end{array}
$$
where
$$
\begin{array}{ccc}
\overline{l}&=&
    \langle \phi_{z_1}\times \phi_{z_2},\phi_{z_1z_1}\rangle,\\
\overline{m}&=&
    \langle \phi_{z_1}\times \phi_{z_2},\phi_{z_1z_2}\rangle,\\
    \overline{n}&=&
    \langle \phi_{z_1}\times \phi_{z_2},\phi_{z_2z_2}\rangle
\end{array}
$$
are holomorphic functions that do not depend on the choice of the branch of the square root function.

For any $v=a\phi_{z_1}(q)+b\phi_{z_2}(q)\in T_pM$, we have 
$$
II_p(v)=a^2l+2ab m+b^2n.
$$

The direction $v\in T_pM$ is \textbf{asymptotic} if $II_p(v)=0$, equivalently,
$$
a^2\overline{l}+2ab \overline{m}+b^2\overline{n}=0.
$$

Following the same calculations as those for surfaces in the Euclidean space (see, for example, \cite{doCarmo}), the matrix of the shape operator  $-dN_p$, with respect to the basis $\phi_{z_1}, \phi_{z_2}$ of $T_pM$, is
$$
A_p=
\left(
\begin{array}{cc}
E&F\\
F&G
\end{array}
\right)^{-1}\left(
\begin{array}{cc}
	l&m\\
	m&n
\end{array}
\right).
$$

Since $II_p$ is a symmetric bilinear form, its matrix $A_p$ has two eigenvalues $\kappa_1, \kappa_2$ and two orthogonal eigenvectors $e_1, e_2$. We call  $\kappa_1, \kappa_2$ the {\bf principal curvatures} and 
$e_1, e_2$ the {\bf principal directions}. 
A direction $a\phi_{z_1}(q)+b\phi_{z_2}(q)\in T_pM$ is principal if and only if 
\begin{equation}
\left|
\begin{array}{ccc}
b^2&-ab&a^2\\
E&F&G\\
l&m&n
\end{array}
\right|=0,
\end{equation}
equivalently, 
\begin{equation}\label{eq:PrincipalDir}
\left|
\begin{array}{ccc}
b^2&-ab&a^2\\
E&F&G\\
\overline{l}&\overline{m}&\overline{n}
\end{array}
\right|=(F\overline{n}-G\overline{m})b^2 +(E\overline{n}-G\overline{l})ab+(E\overline{m}-F\overline{l})a^2=0.
\end{equation}

We denote by $K=\kappa_1\kappa_2=\det(A_p)$ the 
{\bf Gaussian curvature} of $M$.
We have
$$
K=\frac{ln-m^2}{EG-F^2}.
$$ 

Following the same arguments as for surfaces in the Euclidean 3-space (\cite{doCarmo}), one can show that the Gaussian curvature is an intrinsic property of $M\setminus IL$, and depends only on the first fundamental form. 

The set of points on $M$ where $K(p)=0$ (equivalently, $\overline{l}\overline{n}-\overline{m}^2=0$)
is called the {\bf parabolic set}. 

A point $p$ is called {\bf umbilic} if 
$\kappa_1(p)=\kappa_2(p).$ Umbilic points are also the points where the coefficients of the quadratic equation \eqref{eq:PrincipalDir} vanish.
The vanishing of two of these coefficients implies the vanishing of the third, so
such points are generically isolated.

To each non-umbilic point $p=\phi(q)$ are associated two \textbf{focal points} 
$c_i=\phi(q)+\frac{1}{\kappa_i(q)} N(q)$, $i=1$ or $i=2$.
As $q$ varies, the focal points trace the two sheets of the \textbf{focal set} of $M$. The \textbf{ridge} is the preimage on the surface of the singular set of the focal set.

\begin{rem}
{\rm
Clearly, the concepts of asymptotic and principal directions, as well as parabolic and umbilic points, are independent of the choice of branch of the square root function used to define the Gauss map.
One can also show that the focal points are independent of this choice.
}
\end{rem}

With the choice of the holomorphic metric \eqref{eq:holomorphicscprod}, we can study the contact of the surface $M\setminus IL$ 
with planes (these have zero Gaussian curvature), complex sphere (which have constant Gaussian curvature) and lines. The resulting families of mappings that measure the contact between $M$ with these objects are holomorphic (which is not the case when using the Hermitian inner product). 

%%%%%%%%%%%%%%%%%%%%%%%%%%%%

\subsection{Contact with planes}\label{ss:contactPlanesSurf}
We define, as for plane curves, the family height functions on the image of a local parametrization  $\phi:U\setminus IL\to \mathbb C^3$ of a surface as the holomorphic function
$H:U\setminus IL\times \mathbb CS^2\to \mathbb C$, where
$$
H(q,v)=\langle \phi(q),{v}\rangle.
$$

The height function $H_v(q)=H(q,v)$ is singular at $q$ if and only if $v\parallel N(q)$. 
The singularity is of type $A_1$ if and only if $ln-m^2\ne 0$, that is, $q$ is not a parabolic point. 
We have the following result; see \cite[Chapter 6]{IRRT} for the real case analogue.

\begin{theo}\label{theo:CHeightfctSurf} 
For a generic complex surface  $M$ in $\mathbb C^3$, the height function $H_v$ on $M\setminus IL$ can have local singularities of type $A_1,A_2,A_3$ and these are $\mathcal R^+$-versally unfolded by the family of height functions. 
The singularities occur at a point $p\in M\setminus IL$ when:

$A_1$ $\iff$ $v\parallel N(p) $ and $p$ is not a parabolic point.

$A_2$ $\iff$ $v\parallel  N(p)$, $p$ is parabolic point and the unique  asymptotic direction at $p$ is transverse to the parabolic curve.

$A_3$ $\iff$ $v\parallel  N(p)$, $p$ is parabolic point and the unique  asymptotic direction at $p$ is tangent to the parabolic curve at $p$. 
\end{theo} 

We can choose a suitable coordinate system in $\mathbb C^3$ and parametrised the surface $M$ locally at a given point $p_0\in M$ in Monge form $\phi(z_1,z_2)=(z_1,z_2,f(z_1,z_2))$, for some holomorphic function $f$ with zero 1-jet at the origin. The normal vector to $M$ at the origin is $(0,0,1)$ and nearby vectors in $\mathbb C S^2$ can be parametrised by $(v_1,v_2,1)$ with $(v_1,v_2)\in (\mathbb C^2,0)$. We get the germ of the family of height functions $H:(\mathbb C^2\times   \mathbb C^2,(0,0))\to (\mathbb C,0)$, given by
\begin{equation}\label{eq:CheightSurfMonge}
H(z_1,z_2,v_1,v_2)=z_1v_1+z_2v_2+f(z_1,z_2).
\end{equation}

The algebraic conditions on the Taylor expansion of $f$ at the origin for $H_{(0,0)}$ to have one of the singularities listed in Theorem \ref{theo:CHeightfctSurf}, and for these singularities to be $\mathcal R^+$-versally unfolded by the family $H$ in \eqref{eq:CheightSurfMonge}, 
are as given in \cite[Chapter 6]{IRRT} (interpreting the coefficients in  \cite{IRRT} as complex numbers).

\begin{rem}
{\rm 
The family of height functions can be defined at points on the $IL$. As the normal to the surface at such points
is isotropic, we define locally the family of height functions as 
$H:U\times \mathbb CP^2\to \mathbb C$, where
$H(q,\bar{v})=\langle \phi(q),v\rangle$ and $v$ is any representative in $\mathbb C^3$ of $\bar{v}$.
}
\end{rem}
%%%%%%%%%%%%%%%%%%%%%%%%%%%%

\subsection{Contact with spheres}\label{ss:contactSpheSurf}

The contact of the surface  $M=\phi(U)$ at $p$ with the complex sphere of centre ${c}$ passing through $p$ is measured by the singularities of the (contact) function 
\begin{equation}
	d_{c}(q)=\langle \phi(q)-c, {\phi(q)-c}\rangle.
\end{equation}

The function $d_c$ is holomorphic, and we call it the {\bf distance squared function}. 

The family of distance squared functions
$d:U\times  \mathbb C^3\to \mathbb C$ is given by 
\begin{equation}
	d(q,c)=d_c(q)=\langle \phi(q)-c, {\phi(q)-c}\rangle.
\end{equation}

The function $d_c$ measures the contact of $M$ with the complex spheres of centre $c$.
We have the following result, analogous to that for surfaces in the Euclidean 3-space; see \cite[Chapter 6]{IRRT} for the real case analogue.

\begin{theo}\label{theo:Cd2fctSurf} 
	For a generic complex surface  $M$ in $\mathbb C^3$, the distance squared function $d_c$ on $M\setminus IL$ can have local singularities of type $A_1,A_2,A_3,A_4,D_4$ and these are $\mathcal R^+$-versally unfolded by the family of distance squared functions. 
	
	The singularities occur at a point $p=\phi(q)\in M\setminus IL$ when:
	
	$A_1:$ $c$ is on the normal line of $M$ at $p$.
	
	$A_2:$ $c$ is a focal point.
	
	$A_3:$  $c$ is a focal point and is a generic point on the ridge curve.
	
	$A_4:$ $c$ is a special point on the ridge curve.
	
	$D_4:$ $\kappa_1(q)=\kappa_2(q)$, that is, $p$ is a umbilic, and $c=1/{\kappa_1(p)} $. 
\end{theo} 

The structure of the focal set can be obtained from the fact that it is the local component of the bifurcation set of the family of distance-squared functions. Consequently, it is smooth at an $A_2$-singularity of $d_c$, diffeomorphic to a cuspidal edge at an $A_3$-singularity, and to a swallowtail at an $A_4$-singularity. 

As in the case of plane curves, the focal set of a surface $M\setminus IL$ extends to the $IL$. The family of distance-squared functions is well defined at points on the $IL$, and its bifurcation set provides an extension of the focal set to the $IL$. We obtain the following results, analogous to those for surfaces in the Minkowski 3-space \cite{causticsR31}.

\begin{prop}
Let $M$ be a generic complex surface in $\mathbb{C}^{3}$. Then one sheet of its focal set extends to a smooth complex surface along the $IL$ and is tangent to $M$ at that locus. The other sheet of the focal set is either smooth or a cuspidal edge surface.
\end{prop}

If we take $M$ locally  in Monge form (as in \S\ref{ss:contactPlanesSurf}),
we get the germ of the family of distance squared functions
$d:(\mathbb C^2\times   \mathbb C^3,(0,(0,0,c_0)))\to (\mathbb C,0)$, given by
\begin{equation}\label{eq:Cd2SurfMonge}
	d(z_1,z_2,(a,b,c))=(z_1-a)^2+(z_1-a)^2+(f(z_1,z_2)-(c+c_0))^2.
\end{equation}

The algebraic conditions on the Taylor expansion of $f$ at the origin for $d_{(0,0,c_0)}$ to have one of the singularities listed in Theorem \ref{theo:Cd2fctSurf}, and for these singularities to be $\mathcal R^+$-versally unfolded by the family $d$ in \eqref{eq:Cd2SurfMonge}, are as given in \cite[Chapter 6]{IRRT} (interpreting the coefficients in  \cite{IRRT} as complex numbers).

%%%%%%%%%%%%%%%%%%%%%%%%%%%%
\subsection{Contact with lines}\label{ss:contactLinesSurf}

The family of orthogonal projections in $\mathbb C^3$ is as given in \S\ref{ss:contactLinesCurves} and is as follows:
$$
\begin{array}{cccl}
	P:&\mathbb C^3\times \mathbb CS^2&\to& T\mathbb CS^2\\
	&(p,v)&\mapsto&P(p,v)=(v,P_v(p))=(v, p-\langle p,{v}\rangle v),
\end{array}
$$

\begin{theo}\label{theo:CProSurf} 
	For a generic complex surface  $M$ in $\mathbb C^3$, the orthogonal projection 
	$P_v$ on $M\setminus IL$ can have local singularities 
	of $\mathcal A_e$-codimension $\le 2$ and these are $\mathcal A_e$-versally unfolded by the family $P$of orthogonal projections on $M\setminus IL$. 
	\end{theo}

Take $M$ locally  in Monge form (as in \S\ref{ss:contactPlanesSurf}) and project along directions close to $v_0=(0,1,0)$ 
to a fixed plane orthogonal to $v_0$. The directions are parametrised by $(v_1,1,v_2)$ and 
the germ of the family of orthogonal projections  can be taken as 
$P:(\mathbb C^2\times   \mathbb C^2,(0,0))\to (\mathbb C^2,0)$, and given, after a change of coordinate in the source, by
\begin{equation}\label{eq:CProSurfMonge}
	P(z_1,z_2,v_1,v_2)=(z_1, f(z_1+v_1z_2,z_2)-v_2z_2).
\end{equation}

Here too, the algebraic conditions on the Taylor expansion of $f$ at the origin for $P_{(v_1,v_2)}$ to have one of $\mathcal A_e$-codimension $\le 2$ singularities of map-germs from the plane to the plane, and for these singularities to be $\mathcal A_e$-versally unfolded by the family $P$ in \eqref{eq:Cd2SurfMonge}, are as given in \cite[Chapter 6]{IRRT} (interpreting the coefficients in  \cite{IRRT} as complex numbers).

\vspace{0.5cm}

%%%%%%%%%%%%%%%%%%%%%%%%%%%%%%%%%%%%%%%%%%%%%%%%%%%%%%%%%%%%%%%%	       
\noindent
{\bf Acknowledgement}:  
The authors thank Bill Bruce, Igor Mencattini, Juan Nuño Ballesteros, Toru Ohmoto and Raul Oset Sinha for valuable conversations and suggestions.

\vspace{0.5cm}
%%%%%%%%%%%%%%%%%%%%%%%%%%%%%%%%%%%%%%%%%%%%%%%%%%%%%%%%%%%%%%%%	
\noindent
\textbf{Declarations:}

\noindent
Funding : ADF was supported by São Paulo Research Foundation (FAPESP), process number 2023/11669-0.
FT was supported by the FAPESP Thematic project grant 2019/07316-0. 

\noindent
Conflict of interest/Competing interests : Not applicable.

\noindent
Ethics approval : Not applicable.

\noindent
Consent to participate : Not applicable.

\noindent
Consent for publication : Not applicable.

\noindent
Availability of data and materials : Not applicable.

\noindent
Code availability : Not applicable.

\noindent
Authors’ contributions : The authors contributed equally to this work.

% % % % % % % % % % % % % % % % % % % % % % % % % % % % % % % %

%%%%%%%%%%%%%%%%%%%%%%%%%%%%%%%%%%%%%%%%%%%%%%%%%%%%%%%%%%%%%
	
\noindent
Instituto de Ciências Matemáticas e de Computação - USP, Avenida Trabalhador são-carlense, 400, Centro, CEP: 13566-590, São Carlos - SP, Brazil.
\\
E-mails: 
\\
amandafalqueto@gmail.com\\
faridtari@icmc.usp.br

\end{document}